\documentclass[12pt]{amsart}
\usepackage[colorlinks=true]{hyperref}
\usepackage[all]{xy}
\usepackage{amsmath,amssymb,amsthm}
\usepackage{txfonts}
\usepackage{enumerate}

\usepackage{color}

\theoremstyle{plain}
	\newtheorem{thm}{Theorem}[section]
	\newtheorem{prp}[thm]{Proposition}
	\newtheorem{lem}[thm]{Lemma}
	\newtheorem{cor}[thm]{Corollary}
\theoremstyle{definition}
	\newtheorem{dfn}[thm]{Definition}
	\newtheorem{ex}[thm]{Example}
\theoremstyle{remark}
	\newtheorem{rem}[thm]{Remark}

\newcommand{ \Map}{\operatorname{Map}}

\newcommand{ \cat}{\operatorname{cat}}

\newcommand{ \ch}{\operatorname{ch}}
\newcommand{ \hocolim}{\operatorname{hocolim}}

\newcommand{ \colim}{\operatorname{colim}}

\newcommand{ \id}{\operatorname{id}}

\newcommand{ \ad}{\operatorname{ad}}
\newcommand{ \SU}{\operatorname{SU}}
\newcommand{ \Sp}{\operatorname{Sp}}

\newcommand{ \Spin}{\operatorname{Spin}}

\begin{document}
\title{Higher homotopy commutativity in localized Lie groups and gauge groups}
\author{Sho Hasui}
\address{Institute of Mathematics, University of Tsukuba, Ibaraki, 305-8571, Japan}
\email{s.hasui@math.tsukuba.ac.jp}
\author{Daisuke Kishimoto}
\address{Department of Mathematics, Kyoto University, Kyoto, 606-8502, Japan}
\email{kishi@math.kyoto-u.ac.jp}
\author{Mitsunobu Tsutaya}
\address{Faculty of Mathematics, Kyushu University, Fukuoka 819-0395, Japan}
\email{tsutaya@math.kyushu-u.ac.jp}
\thanks{M.T. is supported by Grant for Scientific Research Projects from The Sumitomo Foundation and by JSPS KAKENHI (16K17592).}
\date{}
\begin{abstract}
The first aim of this paper is to study the $p$-local higher homotopy commutativity of Lie groups in the sense of Sugawara.
The second aim is to apply this result to the $p$-local higher homotopy commutativity of gauge groups.
Although the higher homotopy commutativity of Lie groups in the sense of Williams is already known, the higher homotopy commutativity in the sense of Sugawara is necessary for this application.
The third aim is to resolve the $5$-local higher homotopy non-commutativity problem of the exceptional Lie group $\mathrm{G}_2$, which has been open for a long time.
\end{abstract}
\maketitle

\section{Introduction}
\label{section_intro}

Let $G$ be a compact connected Lie group.
It is well known that the $p$-localization $G_{(p)}$ decomposes into a product of spaces such that the number of the factor spaces is not larger than the rank of $G$ and the factor spaces become $p$-local spheres as $p$ gets large enough.
Then we can say that the homotopy type of $G_{(p)}$ becomes simpler as $p$ gets larger.
Now it is natural to ask how the multiplication of $G_{(p)}$ changes as $p$ grows.
McGibbon \cite{MR745146} determined the exact values of $p$ such that $G_{(p)}$ is homotopy commutative. In particular, it turned out that $G_{(p)}$ becomes homotopy commutative if $p$ gets large enough, so as far as we consider homotopy commutativity, we can say that the multiplication of $G_{(p)}$ becomes simpler as $p$ grows. One way to refine McGibbon's work is to consider the higher homotopy commutativity, that is, to consider how high the homotopy commutativity of $G_{(p)}$ gets as $p$ grows. Saumell \cite{MR1337215} went along this line to study the multiplication of $G_{(p)}$ and showed that the homotopy commutativity of $G_{(p)}$ gets higher as $p$ grows.

There are two major definitions of higher homotopy commutativities;
one is \textit{Williams $C_k$-space} \cite{MR0240818} and the other is \textit{Sugawara $C_k$-space} \cite{MR0120645,MR999733}.
The definition of Williams $C_k$-space is done by explicit conditions on higher homotopies parametrized by permutohedra, so it is somewhat intuitive.
On the other hand, the definition of Sugawara $C_k$-space is rather obstruction theoretic, so it is more applicable to practical problems.
There is an implication \cite[Proposition 6]{MR999733}
\[
	\text{Sugawara }C_k\text{-space }\Rightarrow
	\text{Williams }C_k\text{-space }.
\]
For the converse, there is no known implication of Williams $C_k$-space on Sugawara $C_l$-space in general even if $k\ne l$.
The only known counterexample for the converse is the case when $k=\ell=\infty$ \cite[Example 5]{MR999733}.

In the above mentioned result of Saumell, the higher homotopy commutativity is chosen to be the one in the sense of Williams, so it does not imply the one in the sense of Sugawara.
To state the results of McGibbon and Saumell, we need to recall the definition of the type of a Lie group.
Given a compact connected Lie group $G$, the rational cohomology is the exterior algebra
\[
	H^\ast(G;\mathbb{Q})=\Lambda_{\mathbb{Q}}(x_1,\ldots,x_\ell)
\]
by the Hopf theorem, where $x_i\in H^{2n_i-1}(G;\mathbb{Q})$ and $n_1\le\cdots\le n_\ell$.
We call the sequence of the numbers $\{n_1,\ldots,n_\ell\}$ the \textit{type} of the Lie group $G$.
\begin{thm}[McGibbon and Saumell]
\label{thm_McG_Sau}
Given a compact connected simple Lie group $G$ of type $\{n_1,\ldots,n_\ell\}$, a prime $p$ and an integer $k\ge2$, the following assertions hold.
\begin{enumerate}
\item
If $p>kn_\ell$, then $G_{(p)}$ is a Williams $C_k$-space.
\item
If $p<kn_\ell$, then $G_{(p)}$ is not a Williams $C_k$-space, except in the case when $(G,p,k)$ is $(\Sp(2),3,2)$ or $(\mathrm{G}_2,5,k)$ such that $k\le4$.
\end{enumerate}
\end{thm}
The first aim of this paper is to refine McGibbon's result by considering the higher homotopy commutativity in the sense of Sugawara.
\begin{thm}
\label{mainthmA}
Let $G$ be a compact connected Lie group $G$ of type $\{n_1,\ldots,n_\ell\}$, $p$ a prime and $k$ a positive integer.
If $p>kn_\ell$, then the $p$-localization $G_{(p)}$ is a Sugawara $C_k$-space.
\end{thm}
In the proof, we analyze the $A_k$-type of $G$ in the sense of Stasheff \cite{MR0158400}.
The key property of $G$ is that $G$ has the $p$-local $A_k$-type of the product of spheres (Proposition \ref{prp_product1}).

Let $P\to B$ be a principal $G$-bundle.
The \textit{gauge group} $\mathcal{G}(P)$ of $P$ is the topological group consisting of bundle maps $P\to P$ covering the identity on $B$.
For the homotopy commutativity of gauge groups, little is known.
For example, see \cite{MR1348817,MR3082304}.
The second aim of this paper is to study the higher homotopy commutativity of gauge groups in both the sense of Sugawara and Williams by applying Theorem \ref{mainthmA}.
We stress that the higher homotopy commutativity in the sense of Williams is not sufficient for this application.
Let $EG\to BG$ be the universal bundle of $G$ and $E_nG\to B_nG$ be the restriction over the $n$-th projective space $B_nG\subset BG$.
\begin{thm}
\label{mainthmB}
Let $G$ be a compact connected simple Lie group of type $\{n_1,\ldots,n_\ell\}$ and $p$ a prime.
Then, given positive integers $n$ and $k$, the following assertions hold.
\begin{enumerate}
\item
If $p>(n+k)n_\ell$, then $\mathcal{G}(E_nG)_{(p)}$ is a Sugawara $C_k$-space.
\item
If $(n+1)n_\ell<p<(n+k)n_\ell$, then $\mathcal{G}(E_nG)_{(p)}$ is not a Williams $C_k$-space.
\end{enumerate}
\end{thm}

\begin{rem}
Since the gauge group $\mathcal{G}(P)$ need not be connected, we define its $p$ localization by $\mathcal{G}(P)_{(p)}=\Omega(B\mathcal{G}(P)_{(p)})$.
\end{rem}

To prove this theorem, we introduce a new higher homotopy commutativity \textit{$C(k_1,\ldots,k_r)$-space} which is a generalization of $C(k,\ell)$-space \cite{MR2678992}.
This result proves the conjecture by the third author \cite[Conjecture 7.8]{MR3491849} for simple Lie groups.
For general principal bundles, we show the following.
\begin{thm}
\label{mainthmB'}
Let $G$ be a compact connected simple Lie group of type $\{n_1,\ldots,n_\ell\}$ and $p$ a prime.
Given a principal $G$-bundle $P$ over a connected finite complex $B$, the $p$-localized gauge group $\mathcal{G}(P)_{(p)}$ is a Sugawara $C_k$-space if $p>(\cat B+k)n_{\ell}$.
\end{thm}
When $B$ is a sphere, this criterion is not sharp.
We also show the following better criterion which refines the result of Kishimoto--Kono--Theriault \cite{MR3082304}.
\begin{thm}
\label{mainthmC}
Let $G$ be a compact connected simple Lie group of type $\{n_1,\ldots,n_\ell\}$ and $p$ a prime.
If $p\ge kn_\ell+n_i$, then the $p$-localized gauge group $\mathcal{G}(P)_{(p)}$ of any principal $G$-bundle $P$ over $S^{2n_i}$ is a Sugawara $C_k$-space.
\end{thm}

In Theorem \ref{thm_McG_Sau} (2), there are exceptional cases for $\Sp(2)_{(3)}$ and $(\mathrm{G}_2)_{(5)}$.
$\Sp(2)_{(3)}$ and $(\mathrm{G}_2)_{(5)}$ are known to be homotopy commutative \cite{MR745146}.
But the remaining cases for $(\mathrm{G}_2)_{(5)}$ has been open.
The third aim of this paper is to resolve this problem.
\begin{thm}
\label{mainthmD}
The localized Lie group $(\mathrm{G}_2)_{(5)}$ is not a Williams $C_3$-space.
\end{thm}
This result provides a counterexample to the conjecture about the higher homotopy commutativity of the $S^{2p+1}_{(p)}$-bundle $B_1(p)$ over $S^3_{(p)}$ by Hemmi \cite[p.107]{MR1099785}.

This paper is organized as follows.
In Section \ref{section_A_n}, we recall $A_n$-spaces and $A_n$-maps.
In Section \ref{section_highercomm}, we study the characterizations and properties of Sugawara $C_k$-spaces and $C(k_1,\ldots,k_r)$-spaces. 
In Section \ref{section_thmA}, we investigate the $A_k$-types of localized compact connected simple Lie groups.
Theorem \ref{mainthmA} is also shown there.
In Section \ref{section_gauge}, we recall the theory of gauge groups.
In Section \ref{section_thmBC}, we study the higher homotopy commutativity of gauge groups and prove Theorems \ref{mainthmB}, \ref{mainthmB'} and \ref{mainthmC}.
In Section \ref{section_G2}, we prove Theorem \ref{mainthmD} by computing Chern characters.

\section{Higher homotopy associativity}
\label{section_A_n}

In this section, we recall the theory of higher homotopy associativity we need in this paper.
Higher homotopy associativity is formulated by Stasheff \cite{MR0158400}.
To describe it, we need the \textit{associahedra} $\mathcal{K}_2$, $\mathcal{K}_3$ ,...
The $i$-th associahedron $\mathcal{K}_i$ is homeomorphic to the $(i-2)$-dimensional disk.
The boundary sphere is exactly the union of the images of the boundary maps
\[
	\partial_k\colon\mathcal{K}_r\times\mathcal{K}_s\to\mathcal{K}_i
\]
for $r+s-1=i$ and $1\le k\le r$, each of which is an embedding into the boundary.
The degeneracy maps
\[
	s_k\colon\mathcal{K}_i\to\mathcal{K}_{i-1}
\]
for $1\le k\le i$ are also defined.
For details, see \cite{MR0158400}.

\begin{dfn}
Let $G$ be a based space.
Then a family of maps $\{m_i\colon\mathcal{K}_i\times G^{\times i}\to G\}_{i=2}^n$ is said to be an \textit{$A_n$-form} on $X$ if the following conditions are satisfied:
\begin{enumerate}
\item
$m_2(\ast,x)=m_2(x,\ast)=x$,
\item
$m_{r+s-1}(\partial_k(\rho,\sigma);x_1,\ldots,x_{r+s-1})=m_r(\rho;x_1,\ldots,x_{k-1},m_s(\sigma;x_k,\ldots,x_{k-s+1}),x_{k-s},\ldots,x_{r+s-1})$,
\item
$m_i(\rho;x_1,\ldots,x_i)=m_{i-1}(s_k\rho;x_1,\ldots,x_{k-1},x_{k+1},\ldots,x_i)$ if $x_k=\ast$.
\end{enumerate}
A pair $(G,\{m_i\})$ of a based space $G$ and an $A_n$-form $\{m_i\}$ on it is called an \textit{$A_n$-space}.
\end{dfn}

We also recall $A_n$-maps between $A_n$-spaces \cite{MR1000378}.
In the definition, we need the \textit{mulitiplihedra} $\mathcal{J}_1$, $\mathcal{J}_2$,...
The $i$-th multiplihedron is homeomorphic to the $(i-1)$-dimensional disk.
The boundary sphere is exactly the union of the images of the boundary maps
\[
	\delta_k\colon\mathcal{J}_r\times\mathcal{K}_s\to\mathcal{J}_i
\]
for $r+s-1=i$ and $1\le k\le r$ and
\[
	\delta\colon\mathcal{K}_r\times\mathcal{J}_{s_1}\times\cdots\times\mathcal{J}_{s_r}\to\mathcal{J}_i
\]
for $s_1+\cdots+s_r=i$, each of which is an embedding into the boundary.
The degeneracy maps
\[
	d_k\colon\mathcal{J}_i\to\mathcal{J}_{i-1}
\]
for $1\le k\le i$ are also defined.
For details, see \cite{MR1000378}.

\begin{dfn}
Let $(G,\{m_i\})$ and $(G',\{m'_i\})$ be $A_n$-spaces and $f\colon G\to G'$ a based map.
Then a family of maps $\{f_i\colon\mathcal{J}_i\times G^{\times i}\to G'\}_{i=1}^n$ is said to be an \textit{$A_n$-form} on $f$ if the following conditions are satisfied:
\begin{enumerate}
\item
$f_1=f$,
\item
$f_{r+s-1}(\delta_k(\rho,\sigma);x_1,\ldots,x_{r+s-1})=f_r(\rho;x_1,\ldots,x_{k-1},m_s(\sigma;x_k,\ldots,x_{k-s+1}),x_{k-s},\ldots,x_{r+s-1})$,
\item
$f_{s_1+\cdots+s_r}(\delta(\rho,\sigma_1,\ldots,\sigma_r);x_1,\ldots,x_{s_1+\cdots+s_r})$\\
$=m'_r(\rho;f_{s_1}(\sigma_1;x_1,\ldots,x_{s_1}),\ldots,f_{s_r}(\sigma_r;x_{s_1+\cdots+s_{r-1}+1},\ldots,x_{s_1+\cdots+s_r}))$,
\item
$f_i(\rho;x_1,\ldots,x_i)=f_{i-1}(d_k\rho;x_1,\ldots,x_{k-1},x_{k+1},\ldots,x_i)$ if $x_k=\ast$.
\end{enumerate}
A pair $(f,\{f_i\})$ of a based map $f$ and an $A_n$-form $\{f_i\}$ on it is called an \textit{$A_n$-map}.
In particular, if the underlying map of an $A_n$-map is a homotopy equivalence, it is said to be an \textit{$A_n$-equivalence}.
\end{dfn}

If $(f,\{f_i\})$ is an $A_n$-equivalence between non-degenerately based $A_n$-spaces $G$ and $H$, then the homotopy inverse of $f$ also admits an $A_n$-form \cite{MR1000378}.
The following lemma is not difficult to prove.

\begin{lem}
\label{lem_homotopyinv_A_n-map}
Let $(G,\{m_i\})$ and $(G',\{m'_i\})$ be $A_n$-spaces and $(f,\{f_i\})\colon G\to G'$ an $A_n$-map.
If $f'\colon G\to G'$ is a based map homotopic to $f$, then there is an $A_n$-form $\{f'_i\}$ on $f'$ such that $(f',\{f'_i\})$ is homotopic to $(f,\{f_i\})$ as an $A_n$-map.
\end{lem}

If $(G,\{m_i^G\})$ and $(H,\{m^H_i\})$ are $A_n$-spaces, the product space $G\times H$ admits the \textit{product $A_n$-form} $\{m_i^{G\times H}\}$ defined by
\[
	m_i^{G\times H}(\rho;(x_1,y_1),\ldots,(x_i,y_i))=(m_i^G(\rho;x_1,\ldots,x_i),m_i^H(\rho;y_1,\ldots,y_i)).
\]
We call $(G\times H,\{m_i^{G\times H}\}_i)$ the \textit{product $A_n$-space} of $(G,\{m_i^G\})$ and $(H,\{m^H_i\})$.

Stasheff introduced \cite{MR0158400} the \textit{$A_n$-structure} of an $A_n$-space, which is a kind of iterated Dold--Lashof construction or partial universal principal bundle.
We reformulate it as follows.

\begin{dfn}[Stasheff]
Given a based space $G$, the following data is called an \textit{$A_n$-structure} on $G$:
\begin{enumerate}[(i)]
\item
a commutative ladder of based spaces
\[
\xymatrix{
	E_0 \ar[r] \ar[d]
		& E_1 \ar[r] \ar[d]
		& \cdots \ar[r]
		& E_{n-1} \ar[d] \\
	B_0 \ar[r]
		& B_1 \ar[r]
		& \cdots \ar[r]
		& B_{n-1},
}
\]
where $B_0$ is contractible,
\item
a homotopy equivalence $\eta\colon G\to E_0$,
\item
a factorization $E_{i-1}\xrightarrow{h_0}D_{i-1}\xrightarrow{h}E_i$ through a contractible space $D_{i-1}$ of the above map $E_{i-1}\to E_i$ for each $i$.
\end{enumerate}
We say that the $A_n$-structure is \textit{cofibrant} if the basepoint of $G$ is non-degenerate, each $h_0$ is a cofibration and the induced map
\[
	B_{i-1}\cup_{E_{i-1}}D_{i-1}\to B_i
\]
from the pushout is a homeomorphism.
We say that the $A_n$-structure is \textit{fibrant} if each map $E_i\to B_i$ is a fibration and each square in the condition (i) is a pullback.
\end{dfn}

\begin{rem}
While we used the terms \textit{cofibrant} and \textit{fibrant}, we do not insist on the existence of any model category structures of $A_n$-structures.
\end{rem}

\begin{dfn}
Given $A_n$-structures $\{E_i,B_i,D_i,\eta,h_0,h\}$, $\{E'_i,B'_i,D'_i,\eta',h'_0,h'\}$ of $G$, $G'$ and a based map $f\colon G\to G'$, a family of maps
\[
	f^E\colon E_{i-1}\to E'_{i-1}\qquad f^B\colon B_{i-1}\to B'_{i-1}\qquad\text{and}\qquad f^D\colon D_{i-1}\to D'_{i-1}
\]
is said to be an \textit{$A_n$-structure} on $f$ or a map between these $A_n$-structures if the following conditions are satisfied:
\begin{enumerate}[(i)]
\item
these maps satisfies $f^E(E_i)\subset E'_i$, $f^B(B_i)\subset B'_i$ and $f^D(D_i)\subset D'_i$ for each $i$ and the following diagram commutes:
\[
\xymatrix{
	E_{i-2}	\ar[r] \ar[d]^-{f^E}
		& D_{i-2} \ar[r] \ar[d]^-{f^D}
		& E_{i-1} \ar[r] \ar[d]^-{f^E}
		& B_{i-1} \ar[d]^-{f^B}
		& B_{i-2} \ar[l] \ar[d]^-{f^B} \\
	E'_{i-2}	\ar[r]
		& D'_{i-2} \ar[r]
		& E'_{i-1} \ar[r]
		& B'_{i-1}
		& B'_{i-2} \ar[l]
}
\]
\item
the following diagram commutes up to homotopy:
\[
\xymatrix{
	G \ar[r]^-{\eta} \ar[d]_-{f}
		& E_0 \ar[d] \\
	G' \ar[r]^-{\eta'}
		& E'_0
}
\]
\end{enumerate}
\end{dfn}

If $G$ is an $A_n$-space and the basepoint is non-degenerate, Stasheff \cite{MR0158400} constructed a cofibrant $A_n$-structure
\[
\xymatrix{
	E_0G \ar[r] \ar[d]
		& E_1G \ar[r] \ar[d]
		& \cdots \ar[r]
		& E_{n-1}G \ar[d] \\
	B_0G \ar[r]
		& B_1G \ar[r]
		& \cdots \ar[r]
		& B_{n-1}G
}
\]
as a variant of bar construction, where $B_0=\ast$, each $B_{i-1}\to B_i$ is a closed cofibration, $E_0G=G$, $E_{i-1}G$ is contained in a contractible subset $D_{i-1}G$ of $E_iG$, $D_0G$ is the reduced cone of $G$ and each square is a pullback.
We call it the \textit{canonical $A_n$-structure} of $G$.
The space $E_iG$ has the homotopy type of the $(i+1)$-fold join $G^{\ast(i+1)}$ of $G$.
The space $B_iG$ is called the \textit{$i$-th projective space}, where in fact, the $n$-th projective space $B_nG$ is also canonically defined as the mapping cone of $E_{n-1}G\to B_{n-1}G$.
When $n=\infty$, the space $BG=\colim_nB_n G$ is the classifying space of $G$ and $EG=\colim_nE_nG$ is contractible.
We denote the canonical inclusion by $i_k\colon B_kG\to BG$.
Note that each square is a homotopy pullback if $G$ is looplike, where we say an $A_n$-space $(G,\{m_i\})$ $(n\ge2)$ is \textit{looplike} if the left and the right translations in $\pi_0(G)$ induced from $m_2$ are bijections.
Moreover, if an $A_n$-map $G\to G'$ between $A_n$-spaces is given, then there is the canonical map between the canonical $A_n$-structures.
This is obtained by Iwase--Mimura \cite{MR1000378}.
More explicit constructions of these $A_n$-structures can be found in \cite{Iwase}. 

\begin{ex}
If $G$ is a non-degenerately based topological group, then the projection
\[
	EG\to BG
\]
of the canonical $A_\infty$-structure is a principal bundle.
Thus it is fibrant.
\end{ex}

Conversely, Stasheff \cite{MR0158400} also constructed an $A_n$-space from an $A_n$-structure.

\begin{lem}
Let $\{E_i,B_i,D_i,\eta,h_0,h\}$ be an $A_n$-structure of a based space $G$ such that each square
\[
\xymatrix{
	E_{i-1} \ar[r] \ar[d]
		& E_i \ar[d] \\
	B_{i-1} \ar[r]
		& B_i
}
\]
is a homotopy pullback.
Then, there exists a map from $\{E_i,B_i,D_i,\eta,h_0,h\}$ to a fibrant $A_n$-structure $\{\tilde{E}_i,B_i,\tilde{D}_i,\tilde{\eta},\tilde{h}_0,\tilde{h}\}$ on $G$ such that the underlying map is the identity on $G$.
\end{lem}

\begin{proof}
One can find a commutative square
\[
\xymatrix{
	E_{n-1} \ar[r] \ar[d]
		& \tilde{E}_{n-1} \ar[d] \\
	B_{n-1} \ar@{=}[r]
		& B_{n-1}
}
\]
such that $E_{n-1}\to\tilde{E}_{n-1}$ is a closed cofibration and a homotopy equivalence, and $\tilde{E}_{n-1}\to B_{n-1}$ is a fibration.
Take $\tilde{E}_{i-1}\to B_{i-1}$ as the pullback of $\tilde{E}_{n-1}\to B_{n-1}$ along the map $B_{i-1}\to B_{n-1}$ and $\tilde{D}_{i-1}$ the pushout of $\tilde{E}_{i-1}\leftarrow E_{i-1}\rightarrow D_{i-1}$.
By this construction, there are canonical maps $\tilde{E}_{i-1}\xrightarrow{\tilde{h}_0}\tilde{D}_{i-1}\xrightarrow{\tilde{h}}\tilde{E}_i$ and $\tilde{\eta}\colon G\to\tilde{E}_0$.
It is easy to see that $\{\tilde{E}_i,B_i,\tilde{D}_i,\tilde{\eta},\tilde{h}_0,\tilde{h}\}$ is the desired $A_n$-structure.
\end{proof}

We call it the \textit{fibrant replacement} of an $A_n$-structure.

\begin{prp}[Stasheff]
\label{prp_A_n-str_A_n-sp}
Given a fibrant $A_n$-structure $E=\{E_i,B_i,D_i,\eta,h_0,h\}$ of a non-degenerately based space $G$, there exist an $A_n$-form $\{m_i\}$ on $G$ and a map from the canonical $A_n$-structure of $(G,\{m_i\})$ to $E$ of which the underlying map is the identity on $G$.
Moreover, such an $A_n$-space $(G,\{m_i\})$ is looplike.
\end{prp}

For maps between $A_n$-structures, Iwase--Mimura \cite{MR1000378} proved the following proposition.

\begin{prp}[Iwase--Mimura]
\label{prp_IwaseMimura}
Let $G$ and $G'$ be non-degenerately based $A_n$-spaces and suppose $G'$ is looplike.
Denote the canonical $A_n$-structure of $G$ by $E$ and a fibrant replacement of the canonical $A_n$-structure of $G'$ by $\tilde{E}'$.
If a based map $f\colon G\to G'$ admits an $A_n$-structure $E\to\tilde{E}'$, then $f$ admits an $A_n$-form.
\end{prp}

Combining with the fiber-cofiber argument, the following corollary follows.

\begin{cor}
\label{cor_A_n-map_projectivesp}
Let $G$ be a non-degenerately based $A_n$-space and $G'$ be a non-degenerately based looplike $A_\infty$-space.
Then a based map $f\colon G\to G'$ admits an $A_n$-form if and only if the composite
\[
	\Sigma G\xrightarrow{\Sigma f}\Sigma G'\xrightarrow{i_1}BG'
\]
extends over the $n$-th projective space $B_nG$.
\end{cor}

\section{Higher homotopy commutativity}
\label{section_highercomm}

In this section, we study the properties and relations of higher homotopy commutativities.

\subsection{\texorpdfstring{$A_n$-structure on product $A_n$-space}{An-structure on product An-space}}
\label{subsection_product_A_n-space}
The following $A_n$-structure is given by Iwase \cite[Section 4]{MR1642747}.
\begin{lem}
Let $G$ and $H$ be non-degenerately based $A_n$-spaces.
Define spaces $E_i(G,H)$, $B_i(G,H)$ and $D_i(G,H)$ by 
\begin{align*}
	E_i(G,H)&=\bigcup_{j_1+j_2=i}E_{j_1}G\times E_{j_2}H,\\
	B_i(G,H)&=\bigcup_{j_1+j_2=i}B_{j_1}G\times B_{j_2}H,\\
	D_i(G,H)&=\bigcup_{j_1+j_2=i}(D_{j_1}G\times E_{j_2}H\cup\ast\times D_{j_2}H).
\end{align*}
Then the family $\{E_i(G,H),B_i(G,H),D_i(G,H)\}$ is an $A_n$-structure of $G\times H$.
Moreover, if $G$ and $H$ are looplike, the square
\[
\xymatrix{
	E_{i-1}(G,H) \ar[r] \ar[d]
		& E_i(G,H) \ar[d] \\
	B_{i-1}(G,H) \ar[r]
		& B_i(G,H)
}
\]
is a homotopy pullback for each $i$.
\end{lem}

The following proposition plays an important role in the proof of our theorems.

\begin{prp}
\label{prp_product_projectivesp}
Let $G$ and $H$ be non-degenerately based looplike $A_n$-spaces.
Then there is a homotopy commutative diagram
\[
\xymatrix{
	\Sigma(G\times H) \ar[r] \ar[d]_-{\Sigma p_1+\Sigma p_2}
		& B_2(G\times H) \ar[r] \ar[d]
		& \cdots \ar[r]
		& B_n(G\times H) \ar[d] \\
	\Sigma G\vee\Sigma H \ar[r]
		& B_2(G,H) \ar[r]
		& \cdots \ar[r]
		& B_n(G,H),
}
\]
where $p_i$ is the $i$-th projection and the addition is given by the suspension parameter of $\Sigma(G\times H)$.
\end{prp}

\begin{proof}
By Proposition \ref{prp_A_n-str_A_n-sp}, there is an $A_n$-form $\{m'_i\}$ on $G\times H$ such that there is a map between the associated canonical $A_n$-structure to the fibrant replacement $\{\tilde{E}_i(G,H),B_i(G,H),\tilde{D}_i(G,H)\}$.
Since the projections from $\{E_i(G,H),B_i(G,H),D_i(G,H)\}$ to the canonical $A_n$-structures of $G$ and $H$ are $A_n$-structures on $p_1\colon G\times H\to G$ and $p_2\colon G\times H\to H$, the identity map $G\times H\to G\times H$ admits an $A_n$-form $\{f_i\}$ as an $A_n$-map from $(G\times H,\{m'_i\})$ to the product $A_n$-space $G\times H$.
Then, since the pair $(\id,\{f_i\})$ is an $A_n$-equivalence from $(G\times H,\{m'_i\})$ to the product $A_n$-space $G\times H$, the identity map also admits an $A_n$-form as an $A_n$-map from the product $A_n$-space $G\times H$ to $(G\times H,\{m'_i\})$.
Thus we have a map between the canonical $A_n$-structures of them of which the underlying map is the identity on $G\times H$.
Moreover, since the composite
\[
	E_{n-1}(G,H)\to B_{n-1}(G,H)\to B_n(G,H)
\]
is null-homotopic, the map $B_{n-1}(G\times H)\to B_n(G,H)$ extends over $B_n(G\times H)$.
Hence we have a homotopy commutative ladder
\[
\xymatrix{
	\Sigma(G\times H) \ar[r] \ar[d]
		& B_2(G\times H) \ar[r] \ar[d]
		& \cdots \ar[r]
		& B_n(G\times H) \ar[d] \\
	\Sigma G\vee\Sigma H \ar[r]
		& B_2(G,H) \ar[r]
		& \cdots \ar[r]
		& B_n(G,H).
}
\]
By observing the composite
\[
	D_0(G\times H)=C(G\times H)
		\to D_0(G,H)=CG\times H\cup\ast\times CH
		\to\Sigma G\vee\Sigma H,
\]
we can see that the map $\Sigma(G\times H)\to\Sigma G\vee\Sigma H$ is homotopic to $\Sigma p_1+\Sigma p_2$.
\end{proof}

\subsection{\texorpdfstring{Sugawara $C_n$-space}{Sugawara Cn-space}}
\label{subsection_Sugawara_C_n-space}
Let us recall the higher homotopy commutativity introduced by Sugawara \cite{MR0120645} for $n=\infty$ and generalized by McGibbon \cite{MR999733} for $n<\infty$.

\begin{dfn}
An $A_n$-space $G$ is said to be a \textit{Sugawara $C_n$-space} if the multiplication
\[
	m_2\colon G\times G\to G
\]
admits an $A_n$-form as an $A_n$-map which respects the product $A_n$-form on $G\times G$.
\end{dfn}

We give an obstruction theoretic characterization of a Sugawara $C_n$-space.
A similar characterization is obtained by Hemmi--Kawamoto \cite[Corollary 1.1]{MR2793107}.

\begin{prp}
\label{prp_Sugawara_projectivesp}
A looplike $A_\infty$-space $G$ having the based homotopy type of a CW complex is a Sugawara $C_n$-space if and only if the composite
\[
	\Sigma G\vee\Sigma G\to BG\vee BG\xrightarrow{\nabla}BG
\]
of the wedge sum of the inclusions and the folding map extends over the space $B_n(G,G)$.
\end{prp}

\begin{proof}
One can find a topological group $G'$ and an $A_\infty$-equivalence $G'\to G$.
For example, take $G'$ as the geometric realization of Kan's simplicial loop group on $BG$.
An $A_\infty$-equivalence induces a homotopy equivalence between the projective spaces.
Then we may assume that $G$ is a topological group.

By Corollary \ref{cor_A_n-map_projectivesp}, if the multiplication $m\colon G\times G\to G$ is an $A_n$-map, there is a homotopy commutative diagram
\[
\xymatrix{
	\Sigma(G\times G) \ar[r]^-{\Sigma m} \ar[d]_-{i_1}
		& \Sigma G \ar[d]^-{i_1} \\
	B_n(G\times G) \ar[r]^-{\mu}
		& BG.
}
\]
The projections $B(G\times G)\to BG$ induce a homotopy equivalence $B(G\times G)\to BG\times BG$.
Considering the homotopy inverse, we have the factorizations
\[
\xymatrix{
	\Sigma G\vee\Sigma G \ar[r] \ar[d]_-{\text{inclusion}}
		& B_2(G,G) \ar[r] \ar[d]
		& \cdots \ar[r]
		& B_n(G,G) \ar[r] \ar[d]^-{\varphi}
		& BG\times BG \ar[d]^{\simeq} \\
	\Sigma(G\times G) \ar[r]
		& B_2(G\times G) \ar[r]
		& \cdots \ar[r]
		& B_n(G\times G) \ar[r]^-{i_n}
		& B(G\times G)
}
\]
since $B_i(G,G)=B_{i-1}(G,G)\cup_{E_{i-1}(G,G)}D_{i-1}(G,G)$ has the homotopy type of the mapping cone of $E_{i-1}(G,G)\to B_{i-1}(G,G)$.
Thus the composite
\[
	B_n(G,G)\xrightarrow{\varphi}B_n(G\times G)\xrightarrow{\mu}BG
\]
is restricted to a map homotopic to the wedge sum of the inclusions $\Sigma G\vee\Sigma G\to BG$.

Conversely, suppose that there is a map $f\colon B_n(G,G)\to BG$ which is an extension of the wedge sum of the inclusions $\Sigma G\vee\Sigma G\to BG$ and $n\ge2$.
By Proposition \ref{prp_product_projectivesp} and Corollary \ref{cor_A_n-map_projectivesp}, there is a map $m'\colon G\times G\to G$ admitting an $A_n$-form such that $m'$ restricts to the folding map $G\vee G\to G$.
Since $m'$ admits an $A_2$-form, the two maps
\[
	(x_1,x_2,y_1,y_2)\mapsto m'(m(x_1,x_2),m(y_1,y_2)),
		\qquad (x_1,x_2,y_1,y_2)\mapsto m(m'(x_1,y_1),m'(x_2,y_2))
\]
are homotopic.
Then by the Eckmann--Hilton argument, $m$ and $m'$ are homotopic.
Therefore, $m$ also admits an $A_n$-form by Lemma \ref{lem_homotopyinv_A_n-map}.
\end{proof}

\subsection{\texorpdfstring{$C(k_1,\ldots,k_r)$-space}{C(k1,...kr)-space}}
\label{subsection_C(k1...kr)}

For our applications to gauge groups, it is convenient to generalize $C(k,\ell)$-space \cite{MR2678992} as follows.

\begin{dfn}
\label{dfn_C(k1...kr)}
A looplike $A_\infty$-space $G$ is said to be a \textit{$C(k_1,\ldots,k_r)$-space} ($r\ge2,k_1,\ldots,k_r\ge1$) if the wedge sum of inclusions
\[
	\Sigma G\vee\cdots\vee \Sigma G\to BG
\]
extends over the product $B_{k_1}G\times\cdots\times B_{k_r}G$.
\end{dfn}

As in \cite[Section 3]{MR1337215}, when $k_1=\cdots=k_r=1$, a $C(k_1,\ldots,k_r)$-space is exactly a Williams $C_r$-space.
When $k_1=\cdots=k_r=\infty$, a $C(k_1,\ldots,k_r)$-space is exactly a $C(\infty,\infty)$-space and hence a Sugawara $C_\infty$-space.
Hemmi--Kawamoto \cite{MR2793107} proved that a Sugawara $C_n$-space is described by explicit higher homotopies using the resultohedra.
Analogously, the authors guess that our new ``commutativity'' is also described by certain polytopes.
But we do not try to do this in the present paper.

The relations with other higher commutativities is obtained as follows.

\begin{prp}
Let $G$ be a looplike $A_\infty$-space having the homotopy type of a CW complex and $r\ge2$ and $k_1,\ldots,k_r\ge1$ be integers.
Then the implications (i)$\Rightarrow$(ii)$\Rightarrow$(iii) hold for the following conditions:
\begin{enumerate}[(i)]
\item
$G$ is a Sugawara $C_{k_1+\cdots+k_r}$-space,
\item
$G$ is a $C(k_1,\ldots,k_r)$-space,
\item
$G$ is a Williams $C_{k_1+\cdots+k_r}$-space.
\end{enumerate}
\end{prp}

\begin{proof}
To prove the implication (i)$\Rightarrow$(ii), suppose $G$ is a Sugawara $C_{k_1+\cdots+k_r}$-space.
By Proposition \ref{prp_Sugawara_projectivesp}, there is a map
\[
	F\colon B_{k_1+\cdots+k_r}(G,G)\to BG
\]
which restricts to the wedge sum of the inclusions $\Sigma G\vee\Sigma G\to BG$.
Assume that we have a map $f_i\colon B_{k_1}G\times\cdots\times B_{k_i}G\to BG$ which is an extension of the wedge sum of the inclusions for $i<r$.
Since $\cat(B_{k_1}G\times\cdots\times B_{k_i}G)\le k_1+\cdots+k_i$, $f_i$ factors through $B_{k_1+\cdots+k_i}G$ up to homotopy.
We also denote this factorization by $f_i$.
Define a map $g$ as the composite
\[
	B_{k_1+\cdots+k_i}G\times B_{k_{i+1}}G\xrightarrow{\text{inclusion}}B_{k_1+\cdots+k_r}(G,G)\xrightarrow{F}BG.
\]
Then the composite
\[
	g\circ(f_i\times\id)\colon(B_{k_1}G\times\cdots\times B_{k_i}G)\times B_{k_{i+1}}G\to BG
\]
is an extension of the wedge sum of the inclusions.
Thus by induction, $G$ is a $C(k_1,\ldots,k_r)$-space.

To prove the implication (ii)$\Rightarrow$(iii), suppose $G$ is a $C(k_1,\ldots,k_r)$-space.
By Definition \ref{dfn_C(k1...kr)}, there is a map
\[
	F'\colon B_{k_1}G\times\cdots\times B_{k_r}G\to BG
\]
which restricts to the wedge sum of the inclusions $(\Sigma G)^{\vee r}\to BG$.
For each $i$, we see by induction that there is a map $h_i\colon(\Sigma G)^{\times k_i}\to B_{k_i}G$ such that the composite of $h_i$ and the inclusion $B_{k_i}G\to BG$ restricts to the wedge sum of the inclusions $(\Sigma G)^{\vee k_i}\to BG$.
Assume we have a map $h'\colon(\Sigma G)^{\times j}\to B_jG$ for some $j<k_i$ such that the composite of $h'$ and the inclusion $B_jG\to BG$ restricts to the wedge sum of the inclusions $(\Sigma G)^{\vee j}\to BG$.
Then the composite
\[
	(\Sigma G)^{\times j}\times\Sigma G\xrightarrow{h'\times\operatorname{id}}
	B_jG\times\Sigma G\xrightarrow{\text{inclusion}}
	B_{k_i}G\times B_{k_{i'}}G\xrightarrow{\text{inclusion}}
	B_{k_1}G\times\cdots\times B_{k_r}G\xrightarrow{F'}BG,
\]
is an extension of the wedge sum of the inclusions, where we can choose $i'\ne i$ since $r\ge2$ and $k_{i'}\ge1$.
This extension factors through $B_{j+1}G$ since $\cat(\Sigma G)^{\times(j+1)}\le j+1$.
Then we obtain $h_i$ by induction.
Now the composite
\[
	(\Sigma G)^{\times k_1}\times\cdots\times(\Sigma G)^{\times k_r}
	\xrightarrow{h_1\times\cdots\times h_r}
	B_{k_1}G\times\cdots\times B_{k_r}G
	\xrightarrow{F'}BG
\]
is an extension of the wedge sum of the inclusions.
This implies that $G$ is a Williams $C_k$-space.
\end{proof}

\section{Reduction of the projective space}
\label{section_thmA}

The key technique in McGibbon \cite{MR745146} and Saumell's \cite{MR1337215} work is reducing the obstruction problem of $\Sigma G$ to that of the wedge of spheres.
For our problem, we reduce the projective space $B_kG$ to some easier space.
This is the aim of this section.
It can be done by proving that $G$ is $A_k$-equivalent to a product of spheres.
This fact can be considered as a higher version of $p$-regularity.
Once it is done, Theorem \ref{mainthmA} immediately follows.

Let $G$ be a compact connected Lie group of type $\{n_1,\ldots,n_\ell\}$.
In this section, we localize spaces and maps at an odd prime $p\ge n_\ell$ and omit the symbol $(p)$ like $G=G_{(p)}$.
Then $G$ is $A_\infty$-equivalent to the product of compact connected simple Lie groups and a torus.
To prove Theorem \ref{mainthmA}, it is sufficient to consider the case when $G$ is simple.
So we suppose $G$ is simple.

First we determine the homotopy type of the projective spaces of spheres.

\begin{lem}
\label{lem_projsp_sphere}
An odd dimensional sphere $S^{2n-1}$ admits an $A_{p-1}$-form.
The cohomology of the projective space $B_kS^{2n-1}$ for $k<p$ is computed as
\[
	H^\ast(B_kS^{2n-1};\mathbb{Z}_{(p)})=\mathbb{Z}_{(p)}[x]/(x^{k+1}),
\]
where $x\in H^{2n}(B_kS^{2n-1};\mathbb{Z}_{(p)})$.
Moreover, $B_kS^{2n-1}$ has the homotopy type of the CW complex
\[
	S^{2n}\cup e^{4n}\cup\cdots\cup e^{2kn},
\]
where $e^d$ denotes a $d$-dimensional ($p$-local) cell.
\end{lem}

\begin{proof}
This follows from the fact that the homotopy fiber of the double suspension map
\[
E^2\colon S^{2n-1}\to\Omega^2S^{2n+1}
\]
is $(2pn-4)$-connected and $\Omega^2S^{2n+1}$ is an $A_\infty$-space.
\end{proof}

As is well-known, $G$ has the ($p$-local) homotopy type of the product of spheres.
Take generators $\epsilon_i\in\pi_{2n_i-1}(G)$ of the free part of the homotopy groups.
Then the composite
\[
	S^{2n_1-1}\times\cdots\times S^{2n_\ell-1}\xrightarrow{\epsilon_1\times\cdots\times\epsilon_\ell}
		G\times\cdots\times G\xrightarrow{\mathrm{multiplication}}G.
\]
is a homotopy equivalence.
Consider a union of the product of projective spaces
\[
	B_k(S^{2n_1-1},\ldots,S^{2n_\ell-1})=\bigcup_{j_1+\cdots+j_\ell=k}B_{j_1}S^{2n_1-1}\times\cdots\times B_{j_\ell}S^{2n_\ell-1}.
\]

\begin{prp}
\label{prp_product1}
If $p>kn_\ell$ for some $k\ge1$, then the above homotopy equivalence admits an $A_k$-form.
\end{prp}

\begin{proof}
Note that $\pi_i(BG)=0$ for odd $i<2p+1$ since $G\simeq S^{2n_1-1}\times\cdots\times S^{2n_\ell-1}$.
Then, by Lemma \ref{lem_projsp_sphere}, there are no obstructions to extending the map $B_1(S^{2n_1-1},\ldots,S^{2n_\ell-1})\to BG$ over $B_k(S^{2n_1-1},\ldots,S^{2n_\ell-1})$.
Hence by Proposition \ref{prp_product_projectivesp} and Corollary \ref{cor_A_n-map_projectivesp}, the map $S^{2n_1-1}\times\cdots\times S^{2n_\ell-1}\to G$ admits an $A_k$-form.
\end{proof}

The following proposition is used to reduce the projective space $B_kG$ to $B_k(S^{2n_1-1},\ldots,S^{2n_\ell-1})$.

\begin{prp}
\label{prp_product2}
There exists an $A_\infty$-form on $S^{2n_1-1}\times\cdots\times S^{2n_\ell-1}$ such that the restricted $A_k$-form coincides with the product $A_k$-form and the above homotopy equivalence $S^{2n_1-1}\times\cdots\times S^{2n_\ell-1}\to G$ admits an $A_\infty$-form.
\end{prp}

\begin{proof}
By Proposition \ref{prp_product1}, the homotopy equivalence $S^{2n_1-1}\times\cdots\times S^{2n_\ell-1}\to G$ admits an $A_k$-form with respect to the product $A_k$-form of $S^{2n_1-1}\times\cdots\times S^{2n_\ell-1}$.
Since this map is a homotopy equivalence, one can observe that there are no obstructions to constructing $A_\infty$-forms on the map and on $S^{2n_1-1}\times\cdots\times S^{2n_\ell-1}$.
\end{proof}

Let us denote the $A_\infty$-space $S^{2n_1-1}\times\cdots\times S^{2n_\ell-1}$ equipped with the above $A_\infty$-form by $H$.

\begin{proof}[Proof of Theorem \ref{mainthmA}]
Let $G$ be a compact connected simple Lie group of type $\{n_1,\ldots,n_\ell\}$ and take a prime $p$ and a positive integer $k$ such as $p>kn_\ell$.
Then, by Propositions \ref{prp_product_projectivesp}, \ref{prp_Sugawara_projectivesp} and \ref{prp_product2}, $G$ is a Sugawara $C_k$-space if the composite
\[
	B_1(S^{2n_1-1},\ldots,S^{2n_\ell-1})\vee B_1(S^{2n_1-1},\ldots,S^{2n_\ell-1})\to BH\vee BH
		\xrightarrow{\nabla}BH
\]
extends over the union $\bigcup_{k_1+k_2=k}B_{k_1}(S^{2n_1-1},\ldots,S^{2n_\ell-1})\times B_{k_2}(S^{2n_1-1},\ldots,S^{2n_\ell-1})$.
Now it does by Lemma \ref{lem_projsp_sphere} since $\pi_i(BG)=0$ for odd $i<2p+1$ and $p>kn_\ell$.
Thus Theorem \ref{mainthmA} follows.
\end{proof}

\section{Gauge groups}
\label{section_gauge}

In this section, we recall the basic definitions and facts about gauge groups.

\begin{dfn}
Given a principal $G$-bundle $P\to B$, a map $P\to P$ is said to be an \textit{automorphism} if $f$ is $G$-equivariant and induces the identity on $B$.
The topological group consisting of automorphisms on $P$ is denoted by $\mathcal{G}(P)$ and called the \textit{gauge group}.
\end{dfn}

Let $P\to B$ be a principal $G$-bundle.
The associated bundle
\[
	\ad P=(P\times G)/\sim
\]
defined by the equivalence relation
\[
	(ug,x)\sim(u,gxg^{-1})
\]
is called the \textit{adjoint bundle} of $P$.
It is naturally a fiberwise topological group.
Thus the space of sections $\Gamma(\ad P)$ is a topological group.
It is not difficult to see that $\Gamma(\ad P)$ is naturally isomorphic to $\mathcal{G}(P)$.

The weak homotopy type of the classifying space of a gauge group is studied by Gottlieb \cite{MR0309111}.

\begin{prp}
Let $P$ be a principal $G$-bundle over a CW complex $B$, which is classified by a map $\alpha\colon B\to BG$.
Then, the classifying space $B\mathcal{G}(P)$ is weakly homotopy equivalent to the path-component $\Map(B,BG)_\alpha$ of $\Map(B,BG)$ based at $\alpha\in\Map(B,BG)$.
\end{prp}

By \cite[Theorem 3.11, Chapter II]{MR0312509}, if a $p$-localization $\ell\colon X\to X_{(p)}$ of a nilpotent space $X$ is given and $B$ is a finite complex, the induced map $\Map(B,X)_f\to\Map(B,X_{(p)})_{\ell\circ f}$ between the path components containing $f$ and $\ell\circ f$ respectively is also a $p$-localization for any $f\colon B\to X$.
This implies the following corollary.
We recall that even if $\mathcal{G}(P)$ is not path-connected, we define $\mathcal{G}(P)_{(p)}$ as $\Omega(B\mathcal{G}(P)_{(p)})$.

\begin{cor}
Suppose $G$ is a path-connected topological group having the homotopy type of a CW complex.
Let $P$ be a principal $G$-bundle over a finite CW complex $B$, which is classified by a map $\alpha\colon B\to BG$.
Then, the classifying space $B(\mathcal{G}(P)_{(p)})$ is weakly homotopy equivalent to the path-component $\Map(B,BG_{(p)})_{\ell\circ\alpha}$, where $\ell\colon BG\to BG_{(p)}$ is a $p$-localization.
\end{cor}

\section{Proof of Theorems \ref{mainthmB}, \ref{mainthmB'} and \ref{mainthmC}}
\label{section_thmBC}

As in the theorems, let $G$ be a compact connected simple Lie group of type $\{n_1,\ldots,n_\ell\}$, $p$ a prime and $n$, $k$ positive integers.
In this section, we again localize all spaces and maps at $p$ and omit the localization symbol $(p)$.

First we prove that $\mathcal{G}(E_nG)$ is a Sugawara $C_k$-space if $p>(n+k)n_\ell$.
When $k=1$, we have nothing to prove.
Let us consider the case when $k\ge2$.
By Theorem \ref{mainthmA}, $G$ is a $C(k,n)$-space.
Then the wedge sum of the inclusions
\[
	\Sigma G\vee\Sigma G\to BG
\]
extends over the product $B_kG\times B_nG$.
Combining with \cite[Corollary 1.7]{MR2678992}, this implies that the adjoint bundle $\ad E_nG$ is fiberwise $A_k$-equivalent to the trivial bundle $B_nG\times G$.
For the notions of fiberwise $A_n$-theory we need here, see \cite[Section 3]{MR2678992}.
Consider the following homotopy commutative diagram of fiberwise spaces:
\[
\xymatrix{
	\ad E_nG\times_{B_nG}\ad E_nG \ar[rr]^-{\text{multiplication}} \ar[d]_-{\simeq}
		& & \ad E_nG \ar[d]^-{\simeq} \\
	B_nG\times(G\times G) \ar[rr]^-{\text{multiplication}}
		& & B_nG\times G,
}
\]
where the vertical arrows are fiberwise $A_k$-equivalences.
Since $G$ is a Sugawara $C_k$-space, the bottom arrow is a fiberwise $A_k$-map.
Thus we obtain the following lemma.

\begin{lem}
\label{lem_fwSugawara}
The adjoint bundle $\ad E_nG$ is a fiberwise Sugawara $C_k$-space, that is, the fiberwise multiplication
\[
	\ad E_nG\times_{B_nG}\ad E_nG\to\ad E_nG
\]
is a fiberwise $A_k$-map.
\end{lem}

This implies that the multiplication map
\[
	\mathcal{G}(E_nG)\times\mathcal{G}(E_nG)\to\mathcal{G}(E_nG)
\]
is an $A_k$-map.
Hence $\mathcal{G}(E_nG)$ is a Sugawara $C_k$-space.

For a space $B$ such that $\cat B=n$ and a principal $G$-bundle $P$ over $B$, the classifying map $B\to BG$ factors through $B_nG$.
Then by Lemma \ref{lem_fwSugawara}, the gauge group $\mathcal{G}(P)$ is a Sugawara $C_k$-space.
This completes the proof of Theorem \ref{mainthmB'}.

Next, we observe the non-commutativity of $\mathcal{G}(E_nG)$.
We suppose $(n+1)n_\ell<p<(n+k)n_\ell$.
Since $(n+1)n_\ell<p$, the wedge sum of the inclusions
\[
	\Sigma G\vee B_nG\to BG
\]
extends over the product $\Sigma G\times B_nG$.
Taking the adjoint, we obtain the map
\[
	f\colon\Sigma G\to\Map(B_nG,BG)_{i_n}.
\]
Consider the extension problem of the map
\[
	(\Sigma G)^{\vee k}\xrightarrow{(f,\ldots,f)}\Map(B_nG,BG)_{i_n}
\]
over the product $(\Sigma G)^{\times k}$.
If $\mathcal{G}(E_nG)$ is a Williams $C_k$-space, this extends.
Taking the adjoint, we have the map
\[
	(\Sigma G)^{\vee k}\times B_nG\to BG,
\]
which is an extension of the wedge sum of the inclusions $(\Sigma G)^{\vee k}\vee B_nG\to BG$.
This does not extend over the product since $G$ is not a $C(r_1,\ldots,r_k,n)$-space for $r_1=\cdots=r_k=1$.
Therefore, the gauge group $\mathcal{G}(E_nG)$ is not a Williams $C_k$-space.

Now the proof of Theorem \ref{mainthmC} might be obvious.
Let $P$ be a principal $G$-bundle over $S^{2n_i}$ classified by $\alpha\colon S^{2n_i}\to BG$ and $k\ge 2$ an integer satisfying $p\ge kn_\ell+n_i$.
One can prove by the analogous argument that the wedge sum $S^{2n_i}\vee\Sigma G\to BG$ of $\alpha$ and the inclusion extends over the product $S^{2n_i}\times B_kG$.
Then the adjoint bundle $\ad P$ is fiberwise $A_k$-equivalent to the trivial bundle $S^{2n_i}\times G$.
Since $G$ is a Sugawara $C_k$-space, then the fiberwise multiplication
\[
	\ad P\times_{S^{2n_i}}\ad P\to\ad P
\]
is a fiberwise $A_k$-map.
Therefore, the gauge group $\mathcal{G}(P)$ is a Sugawara $C_k$-space.

\section{\texorpdfstring{$5$-local higher homotopy commutativity of $\mathrm{G}_2$}{5-local higher homotopy commutativity of G2}}
\label{section_G2}

In this section, we prove Theorem \ref{mainthmD}. 
Hereafter, we localize all spaces and maps at $p=5$.
McGibbon \cite{MR745146} proved that $\mathrm{G}_2$ is homotopy commutative.
But Saumell \cite{MR1337215} proved that $\mathrm{G}_2$ is not a Williams $C_5$-space.

By the results in \cite[Lecture 4]{MR0251716}, there is a loop map
\[
	E\colon BU\to BU
\]
characterized by the homotopy commutative diagrams
\[
\xymatrix{
	BU \ar[r]^E \ar[d]
		& BU \ar[d]^-{\ch_{n}} \\
	\ast \ar[r]
		& K(\mathbb{Q},2n)
}
\qquad
\xymatrix{
	BU \ar[r]^E \ar[d]^-{\ch_{4n-2}}
		& BU \ar[d]^-{\ch_{4n-2}} \\
	K(\mathbb{Q},8n-4) \ar@{=}[r]
		& K(\mathbb{Q},8n-4)
}
\]
where the left square holds for $n\not\equiv2\,\operatorname{mod}\,4$ and $\ch_n$ denotes the $n$-th universal Chern character.
From this, we have $E^2=E$.
We consider a telescope
\[
	B'=\hocolim(B^2U\xrightarrow{BE}B^2U\xrightarrow{BE}\cdots)
\]
and define a loop space
\[
	B=\Omega B'.
\]
The canonical map $B^2U\to B'$ induces a loop map $\pi\colon BU\to B$.
Note that $B$ also has the homotopy type of a telescope:
\[
	B\simeq\hocolim(BU\xrightarrow{E}BU\xrightarrow{E}\cdots).
\]
We can compute the cohomology group as
\[
	H^\ast(B;\mathbb{Z}_{(5)})=\mathbb{Z}_{(5)}[z_4,z_{12},z_{20},z_{28},\ldots]
\]
such that $\pi^\ast z_{8n-4}=E^\ast c_{4n-2}$ for the Chern class $c_{4n-2}\in H^{8n-4}(BU;\mathbb{Z}_{(5)})$ by the Newton identities.
\begin{lem}
\label{lem_Ec}
The following congruences modulo the ideal $(c_k\mid k\colon\text{odd or }k\ge7)+((c_2,c_6)^2+(c_4))^2$ hold, that is, the following congruences are modulo monomials containing $c_k$ for odd $k$ or $k\ge7$ or $c_2^pc_4^qc_6^r$ for $p+2q+r=4$:
\begin{align*}
	E^\ast c_2\equiv& c_2,
		& E^\ast c_4\equiv& \frac{1}{2}c_2^2,
		& E^\ast c_6\equiv& c_6-c_4c_2+\frac{1}{2}c_2^3,\\
	E^\ast c_8\equiv& c_6c_2,
		& E^\ast c_{10}\equiv& -c_6c_4+\frac{3}{2}c_6c_2^2,
		& E^\ast c_{12}\equiv& \frac{1}{2}c_6^2,\\
	E^\ast c_{14}\equiv& \frac{3}{2}c_6^2c_2,
		& E^\ast c_{16}\equiv& 0,
		& E^\ast c_{18}\equiv& \frac{1}{2}c_6^3.
\end{align*}
\end{lem}

\begin{proof}
These congruences can be verified by the equalities
\begin{align*}
	E^\ast c_{4n-2}=&-\frac{1}{4n-2}((E^\ast c_{4n-4})s_2+\cdots+(E^\ast c_4)s_{4n-6}+s_{4n-2}),\\
	E^\ast c_{4n}=&-\frac{1}{4n}((E^\ast c_{4n-2})s_2+\cdots+(E^\ast c_2)s_{4n-2})
\end{align*}
and Girard's formula
\[
	s_i=\sum_{r_1+2r_2+\cdots+ir_i=i}(-1)^{i+r_1+\cdots+r_i}\frac{i(r_1+\cdots+r_i-1)!}{r_1!\cdots r_i!}c_1^{r_1}\cdots c_i^{r_i}.
\]
\end{proof}

We also need the indecomposables as in the following lemma.
The proof is similar to the previous lemma.

\begin{lem}
\label{lem_Ec_indecomp}
We have the congruence $E^\ast c_{4n-2}\equiv c_{4n-2}$ mod decomposables for any integer $n\ge1$.
\end{lem}

Now we recall elementary properties of the exceptional Lie group $\mathrm{G}_2$.
The following diagram of inclusions commutes:
\[
\xymatrix{
	\SU(3) \ar[r] \ar[d]_-{\text{realification}}
		& \mathrm{G}_2 \ar[d] \ar[rd]^-{\rho} \\
	\Spin(6) \ar[r]
		& \Spin(7) \ar[r]
		& \SU(7)
}
\]
where $\Spin(7)/\mathrm{G}_2\cong S^7$.
As in \cite[Section 4]{MR807104}, the following proposition holds.
\begin{prp}
\label{prp_Watanabe}
The cohomology of $B\mathrm{G}_2$ is computed as
\[
	H^\ast(B\mathrm{G}_2;\mathbb{Z}_{(5)})=\mathbb{Z}_{(5)}[y_4,y_{12}]
\]
such that the following equality holds:
\[
	\rho^\ast(c_i)=
	\begin{cases}
		-y_{2i}				& i=2,6	\\
		\tfrac{1}{4}y_4^2	& i=4	\\
		0					& \text{otherwise.}
	\end{cases}
\]
\end{prp}

\begin{rem}
It is claimed in \cite[Proposition 2.10]{MR3276725} that $\rho^\ast(c_4)=0$, and this is false as above.
However, this is irrelevant to verifying the results of \cite{MR3276725}.
\end{rem}

It is well known that $\Sigma\mathbb{C}P^6$ has the homotopy type of the wedge sum $A\vee S^5\vee S^7\vee S^9$ where $A=S^3\cup e^{11}$.
The composite of the inclusions $A\to\Sigma\mathbb{C}P^6\to\SU(7)$ lifts to $\Spin(7)$.
Moreover, it lifts to $\mathrm{G}_2$ since $\Spin(7)/\mathrm{G}_2\cong S^7$.

\begin{lem}
\label{lem_H(A)_K(A)}
The cohomology of $A$ is computed as 
\[
	H^\ast(A;\mathbb{Z}_{(5)})=\mathbb{Z}_{(5)}\langle x_3,x_{11}\rangle, 
		\qquad x_3\in H^3,\quad x_{11}\in H^{11},
\]
where $x_3$ and $x_{11}$ are the images of the cohomology suspensions $-\sigma(y_4)$ and $-\sigma(y_{12})$ under the induced map of $A\to G$, respectively.
Moreover, the K-theory of $A$ is computed as 
\[
	\tilde{K}(\Sigma A;\mathbb{Z}_{(5)})=\mathbb{Z}_{(5)}\langle g,h\rangle,
		\qquad \ch g=\Sigma x_3+\frac{1}{5!}\Sigma x_{11},
		\quad \ch h=\Sigma x_{11}.
\]
\end{lem}

Consider the wedge sum of the inclusions
\[
	\Sigma A\vee\Sigma A\vee\Sigma A\to B\mathrm{G}_2.
\]
Since $\mathrm{G}_2$ is homotopy commutative, this map extends over the fat wedge $T(\Sigma A,\Sigma A,\Sigma A)$.
Our goal is to see the higher Whitehead product
\[
	\omega\colon\Sigma^2(A\wedge A\wedge A)\to B\mathrm{G}_2
\]
is non-trivial.
Our basic idea is the same as the calculation of Samelson products in quasi-$p$-regular Lie groups in \cite{MR3775355}.
Once this is proved, Theorem \ref{mainthmD} follows from \cite[Theorem-Definition 3.1]{MR1337215}.

Let $j\colon B\mathrm{G}_2\to B$ be the composite
\[
	B\mathrm{G}_2\xrightarrow{B\rho}B\SU(7)\xrightarrow{\text{inclusion}}BU\xrightarrow{\pi}B
\]
and $W$ be the homotopy fiber of $j$.

\begin{lem}
\label{lem_jz}
The following equalities hold:
\begin{align*}
	j^\ast z_4=&-y_4,
	&j^\ast z_{12}=&-y_{12}-\frac{1}{4}y_4^3,
	&j^\ast z_{20}\equiv&-\frac{5}{4}y_{12}y_4^2,
	&j^\ast z_{28}\equiv&-\frac{3}{2}y_{12}^2y_4,
	&j^\ast z_{36}\equiv&-\frac{1}{2}y_{12}^3,
\end{align*}
where $\equiv$ means the congruence modulo $(y_4,y_{12})^4$, namely modulo the monomials $y_4^py_{12}^q$ for $p+q\ge4$.
\end{lem}

\begin{proof}
This lemma immediately follows from Lemma \ref{lem_Ec} and Proposition \ref{prp_Watanabe}.
\end{proof}

\begin{lem}
\label{lem_gen_of_W}
The cohomology of $W$ is computed as
\[
	H^\ast(W;\mathbb{Z}_{(5)})=\mathbb{Z}_{(5)}\langle a_{19},a_{27},a_{35}\rangle\qquad\text{for }\ast<43,
\]
where the transgressions $\tau(a_{19}),\tau(a_{27}),\tau(a_{35})$ with respect to the fibration $W\to B\mathrm{G}_2\to B$ satisfy
\begin{align*}
	\tau(a_{19})\equiv& z_{20}-\tfrac{5}{4}z_{12}z_4^2&&\text{mod }(z_4^5),\\
	\tau(a_{27})\equiv& z_{28}-\tfrac{3}{2}z_{12}^2z_4
		&&\text{mod }(z_{20})+(z_{4},z_{12})^4,\\
	\tau(a_{35})\equiv& z_{36}-\tfrac{1}{2}z_{12}^3&&\text{mod }(z_{20},z_{28})+(z_4,z_{12})^4,
\end{align*}
where the middle congruence is modulo the monomials containing $z_{20}$ or $z_4^pz_{12}^q$ for $p+q=4$ and the bottom congruence is the modulo monomials containing $z_{20}$, $z_{28}$ or $z_4^pz_{12}^q$ for $p+q=4$.
Moreover, the images of $a_{19},a_{27},a_{35}$ under the induced map of $\Omega B\to W$ are the cohomology suspensions $\sigma(z_{20}),\sigma(z_{28}),\sigma(z_{36})$.
\end{lem}

\begin{proof}
This follows from the computation of the cohomology Serre spectral sequence and Lemma \ref{lem_jz}.
\end{proof}

The map $j$ induces the exact sequence
\[
	[\Sigma^2A^{\wedge 3},\Omega B]\to[\Sigma^2A^{\wedge 3},W]\to[\Sigma^2A^{\wedge 3},B\mathrm{G}_2]\xrightarrow{j_\ast}[\Sigma^2A^{\wedge 3},B].
\]
Let us construct an appropriate lift of $\omega\in[\Sigma^2A^{\wedge 3},B\mathrm{G}_2]$ to $[\Sigma^2A^{\wedge 3},W]$.

\begin{lem}
\label{lem_ext_fatwedge}
The extension of the wedge sum of the inclusions $\Sigma A\vee\Sigma A\vee\Sigma A\to B$ over the fat wedge $T(\Sigma A,\Sigma A,\Sigma A)$ is unique up to homotopy.
\end{lem}

\begin{proof}
This follows from the fact that the homotopy groups $\pi_8(B)$, $\pi_{16}(B)$ and $\pi_{24}(B)$ are trivial.
\end{proof}

Define a map $\tilde{\mu}\colon(\Sigma A)^{\times3}\to B$ by the composite
\[
	(\Sigma A)^{\times3}
		\to(B\mathrm{G}_2)^{\times 3}
		\xrightarrow{j^{\times 3}}B^{\times 3}
		\xrightarrow{\text{multiplication}} B.
\]
Then we obtain the homotopy commutative diagram
\[
\xymatrix{
	\Sigma^2A^{\wedge 3} \ar[r] \ar[d]^-{\mu}
		& T(\Sigma A,\Sigma A,\Sigma A) \ar[r] \ar[d]
		& (\Sigma A)^{\times3} \ar[d]^-{\tilde{\mu}} \\
	W \ar[r]
		& B\mathrm{G}_2 \ar[r]^-{j}
		& B,
}
\]
as follows.
The map $T(\Sigma A,\Sigma A,\Sigma A)\to B\mathrm{G}_2$ is an extension of the wedge sum of the inclusions $(\Sigma A)^{\vee3}\to B\mathrm{G}_2$.
Such extension exists since $\mathrm{G}_2$ is homotopy commutative.
By Lemma \ref{lem_ext_fatwedge}, the right square commutes up to homotopy.
The map $\mu\colon\Sigma^2A^{\wedge 3}\to W$ is defined up to homotopy and the left square commutes since the top row is a cofiber sequence and the bottom row is a fiber sequence.
For a precise description of the top cofiber sequence, see \cite{MR0174054}.
Here $\mu$ is the lifting of the map $\omega\colon\Sigma^2A^{\wedge 3}\to B\mathrm{G}_2$.

To check the non-triviality of $\mu$, we first embed $[\Sigma^2A^{\wedge 3},W]$ to an easier module.

\begin{lem}
\label{lem_injectivity_mu}
The map
\[
	[\Sigma^2A^{\wedge 3},W]\to 
		H^{19}(\Sigma^2A^{\wedge 3};\mathbb{Z}_{(5)})
			\oplus H^{27}(\Sigma^2A^{\wedge 3};\mathbb{Z}_{(5)})
			\oplus H^{35}(\Sigma^2A^{\wedge 3};\mathbb{Z}_{(5)})(\cong\mathbb{Z}_{(5)}^{\oplus7})
\]
defined by $f\mapsto(f^\ast(a_{19}),f^\ast(a_{27}),f^\ast(a_{35}))$ is injective.
\end{lem}

\begin{proof}
First we note that the homotopy set $[\Sigma^2A^{\wedge 3},W]$ is isomorphic to the stable homotopy set $\{\Sigma^2A^{\wedge 3},W\}$ since $W$ is $18$-connected and  $\dim\Sigma^2A^{\wedge 3}=35$.
The rationalized map
\[
	[\Sigma^2A^{\wedge 3},W]\otimes\mathbb{Q}\to H^{19}(\Sigma^2A^{\wedge 3};\mathbb{Q})\oplus H^{27}(\Sigma^2A^{\wedge 3};\mathbb{Q})\oplus H^{35}(\Sigma^2A^{\wedge 3};\mathbb{Q})
\]
is an isomorphism by Lemma \ref{lem_gen_of_W}.
Then it is sufficient to show that $[\Sigma^2A^{\wedge 3},W]$ is a free $\mathbb{Z}_{(5)}$-module.
The homotopy groups of $W$ are computed as
\[
	\pi_i(W)=
		\begin{cases}
			\mathbb{Z}_{(5)}	& i=19,27,35 \\
			0					& i=28,36
		\end{cases}
\]
by the approximation by a CW complex $S^{19}\cup e^{27}\cup e^{35}\to W$ and the stable homotopy groups of spheres.
Thus by a skeletal consideration, one can see that $[\Sigma^2A^{\wedge 3},W]$ is $\mathbb{Z}_{(5)}$-free.
\end{proof}

Next we compute the image of $\mu\in[\Sigma^2A^{\wedge 3},W]$ in $H^{19}(\Sigma^2A^{\wedge 3};\mathbb{Z}_{(5)})\oplus H^{27}(\Sigma^2A^{\wedge 3};\mathbb{Z}_{(5)})\oplus H^{35}(\Sigma^2A^{\wedge 3};\mathbb{Z}_{(5)})$.
\begin{prp}
\label{prp_mua}
The following equalities hold:
\[
	\mu^\ast(a_{19})=-\frac{3}{2}b_{19},
		\qquad \mu^\ast(a_{27})=-2b_{27},
		\qquad \mu^\ast(a_{35})=-2b_{35},
\]
where $b_{19},b_{27},b_{35}$ are defined as
\begin{align*}
	b_{19}=&\Sigma^2 x_{11}\otimes x_3\otimes x_3+\Sigma^2 x_3\otimes x_{11}\otimes x_3+\Sigma^2 x_3\otimes x_3\otimes x_{11},\\
	b_{27}=&\Sigma^2 x_{11}\otimes x_{11}\otimes x_3+\Sigma^2 x_{11}\otimes x_3\otimes x_{11}+\Sigma^2 x_3\otimes x_{11}\otimes x_{11},\\
	b_{35}=&\Sigma^2 x_{11}\otimes x_{11}\otimes x_{11}.
\end{align*}
\end{prp}

\begin{proof}
The previous diagram induces the map of cofiber sequences
\[
\xymatrix{
	\Sigma^2A^{\wedge 3} \ar[r] \ar[d]^-{\mu}
		& T(\Sigma A,\Sigma A,\Sigma A) \ar[r] \ar[d]
		& C_1 \ar[r]^-{\partial} \ar[d]
		& \Sigma^3A^{\wedge 3} \ar[d]^-{\Sigma{\mu}} \\
	W \ar[r]
		& B\mathrm{G}_2 \ar[r]
		& C_2 \ar[r]
		& \Sigma W
}
\]
and hence the next homotopy commutative square
\[
\xymatrix{
	C_1 \ar[r]^-{\simeq} \ar[d]
		& (\Sigma A)^{\times 3} \ar[d]^-{\tilde{\mu}} \\
	C_2 \ar[r]
		& B,
}
\]
where the composite $C_1\to(\Sigma A)^{\times 3}\to(\Sigma A)^{\wedge 3}\cong\Sigma^3A^{\wedge 3}$ is homotopic to $\partial$.

For $i=20,28,36$, we have the following commutative diagram:
\[
\xymatrix{
	\tilde{H}^{i-1}(\Sigma^2A^{\wedge 3};\mathbb{Z}_{(5)}) \ar[r]^-{\cong}
		& \tilde{H}^i(C_1;\mathbb{Z}_{(5)})
		& \tilde{H}^i((\Sigma A)^{\times3};\mathbb{Z}_{(5)}) \ar[l]_-{\cong} \\
	H^{i-1}(W;\mathbb{Z}_{(5)}) \ar[u]_-{\mu^\ast} \ar@{^(->}[r]
		& \tilde{H}^i(C_2;\mathbb{Z}_{(5)}) \ar[u]
		& \tilde{H}^i(B;\mathbb{Z}_{(5)}) \ar[u]_-{\tilde{\mu}^\ast} \ar@{->>}[l]
}
\]
The injectivity and the surjectivity in the bottom row follows from the following diagram and the computation of the transgressions in Lemma \ref{lem_gen_of_W}, where the top row is exact.
\[
\xymatrix{
	0 \ar[r]
		& H^{i-1}(W;\mathbb{Z}_{(5)}) \ar@{^(->}[r]
		& \tilde{H}^i(C_2;\mathbb{Z}_{(5)}) \ar@{->>}[r]
		& \tilde{H}^i(B\mathrm{G}_2;\mathbb{Z}_{(5)}) \ar[r]
		& 0 \\
	 & & \tilde{H}^i(B;\mathbb{Z}_{(5)}) \ar[u] \ar@{->>}[r]
	 	& \tilde{H}^i(B\mathrm{G}_2;\mathbb{Z}_{(5)}) \ar@{=}[u]
}
\]
Under the induced map of the multiplication $(BU)^{\times3}\to BU$, the class $E^\ast c_n\in H^{2n}(BU;\mathbb{Z}_{(5)})$ is mapped to
\[
	\sum_{p+q+r=n}(E^\ast c_p)\times(E^\ast c_q)\times(E^\ast c_r)
\]
by the Cartan formula for Chern classes.
From this, we can compute
\begin{align*}
	\tilde{\mu}^\ast(z_{20}-\frac{5}{4}z_{12}z_4^2)=
		&-\frac{3}{2}(\Sigma x_3\otimes\Sigma x_3\otimes\Sigma x_{11}
			+\Sigma x_3\otimes\Sigma x_{11}\otimes\Sigma x_3
			+\Sigma x_{11}\otimes\Sigma x_3\otimes\Sigma x_3),\\
	\tilde{\mu}^\ast(z_{28}-\frac{3}{2}z_{12}^2z_4)=
		&-2(\Sigma x_3\otimes\Sigma x_{11}\otimes\Sigma x_{11}
			+\Sigma x_{11}\otimes\Sigma x_3\otimes\Sigma x_{11}
			+\Sigma x_{11}\otimes\Sigma x_{11}\otimes\Sigma x_3),\\
	\tilde{\mu}^\ast(z_{36}-\frac{1}{2}z_{12}^3)=
		&-2\Sigma x_{11}\otimes\Sigma x_{11}\otimes\Sigma x_{11}.
\end{align*}
Then by the above diagram and Lemma \ref{lem_gen_of_W}, we obtain $\mu^\ast(a_{19}),\mu^\ast(a_{27}),\mu^\ast(a_{35})$ as above.
\end{proof}

Finally we compute the image of the composite
\[
	\Phi\colon[\Sigma^2A^{\wedge 3},\Omega B]\to[\Sigma^2A^{\wedge 3},W]
		\to H^{19}(\Sigma^2A^{\wedge 3};\mathbb{Z}_{(5)})
			\oplus H^{27}(\Sigma^2A^{\wedge 3};\mathbb{Z}_{(5)})
			\oplus H^{35}(\Sigma^2A^{\wedge 3};\mathbb{Z}_{(5)}).
\]

Consider the commutative diagram
\[
\xymatrix{
	[\Sigma^3A^{\wedge 3},B] \ar[r]^-{\cong} \ar[d]_-{z_i}
		& [\Sigma^2A^{\wedge 3},\Omega B] \ar[r] \ar[d]_-{\sigma(z_i)}
		& [\Sigma^2A^{\wedge 3},W] \ar[d]^-{a_{i-1}} \\
	H^{i}(\Sigma^3A^{\wedge 3};\mathbb{Z}_{(5)}) \ar[r]^-{\cong}
		& H^{i-1}(\Sigma^2A^{\wedge 3};\mathbb{Z}_{(5)}) \ar@{=}[r]
		& H^{i-1}(\Sigma^2A^{\wedge 3};\mathbb{Z}_{(5)})
}
\]
for $i=20,28,36$.
Since $\pi_\ast\colon\tilde{K}(\Sigma^3A^{\wedge 3};\mathbb{Z}_{(5)})\cong[\Sigma^3A^{\wedge 3},BU]\to[\Sigma^3A^{\wedge 3},B]$ is an isomorphism, the image of the left vertical arrow coincides with the image of the following map by Lemma \ref{lem_Ec_indecomp}:
\[
	(9!\ch_{10},13!\ch_{14},17!\ch_{18})\colon\tilde{K}(\Sigma^3A^{\wedge 3};\mathbb{Z}_{(5)})
		\to H^{20}(\Sigma^3A^{\wedge 3};\mathbb{Z}_{(5)})
			\oplus H^{28}(\Sigma^3A^{\wedge 3};\mathbb{Z}_{(5)})
			\oplus H^{36}(\Sigma^3A^{\wedge 3};\mathbb{Z}_{(5)}).
\]
Under the K\"{u}nneth isomorphism
\[
	\tilde{K}(\Sigma^3A^{\wedge 3};\mathbb{Z}_{(5)})\cong\tilde{K}(\Sigma A;\mathbb{Z}_{(5)})^{\otimes 3},
\]
we can compute as
\begin{align*}
	9!\ch_{10}(g\otimes g\otimes g)&=9!((\ch_2g)(\ch_6g)(\ch_6g)+(\ch_2g)(\ch_6g)(\ch_2g)+(\ch_2g)(\ch_2g)(\ch_6g))\\
		&=\frac{9!}{5!}(\Sigma x_{11}\otimes\Sigma x_3\otimes\Sigma x_3
			+\Sigma x_3\otimes\Sigma x_{11}\otimes\Sigma x_3
			+\Sigma x_3\otimes\Sigma x_3\otimes\Sigma x_{11}),\\
	13!\ch_{14}(g\otimes g\otimes g)&=13!((\ch_6g)(\ch_6g)(\ch_2g)+(\ch_6g)(\ch_2g)(\ch_6g)+(\ch_2g)(\ch_6g)(\ch_6g))\\
		&=\frac{13!}{5!5!}(\Sigma x_{11}\otimes\Sigma x_{11}\otimes\Sigma x_3
			+\Sigma x_{11}\otimes\Sigma x_3\otimes\Sigma x_{11}
			+\Sigma x_3\otimes\Sigma x_{11}\otimes\Sigma x_{11})\\
	17!\ch_{18}(g\otimes g\otimes g)&=17!(\ch_6g)(\ch_6g)(\ch_6g)\\
		&=\frac{17!}{5!5!5!}\Sigma x_{11}\otimes\Sigma x_{11}\otimes\Sigma x_{11}
\end{align*}
by Lemma \ref{lem_H(A)_K(A)}.
Similarly, we have
\begin{align*}
	&9!\ch_{10}(h\otimes g\otimes g)	=9!\Sigma x_{11}\otimes\Sigma x_3\otimes\Sigma x_3,\\
	&13!\ch_{14}(h\otimes g\otimes g)	=\frac{13!}{5!}(\Sigma x_{11}\otimes\Sigma x_{11}\otimes\Sigma x_3+\Sigma x_{11}\otimes\Sigma x_3\otimes\Sigma x_{11}),\\
	&17!\ch_{18}(h\otimes g\otimes g)	=\frac{17!}{5!5!}\Sigma x_{11}\otimes\Sigma x_{11}\otimes\Sigma x_{11},\\
	&9!\ch_{10}(h\otimes h\otimes g)	=0,\\
	&13!\ch_{14}(h\otimes h\otimes g)	=13!\Sigma x_{11}\otimes\Sigma x_{11}\otimes\Sigma x_3,\\
	&13!\ch_{18}(h\otimes h\otimes g)	=\frac{17!}{5!}\Sigma x_{11}\otimes\Sigma x_{11}\otimes\Sigma x_{11},\\
	&9!\ch_{10}(h\otimes h\otimes h)	=0,\\
	&13!\ch_{14}(h\otimes h\otimes h)	=0,\\
	&13!\ch_{18}(h\otimes h\otimes h)	=17!\Sigma x_{11}\otimes\Sigma x_{11}\otimes\Sigma x_{11}.
\end{align*}
The other terms are analogous.

\begin{proof}[Proof of Theorem \ref{mainthmD}]
Now suppose that $(\mu^\ast(a_{19}),\mu^\ast(a_{27}),\mu^\ast(a_{35}))$ is contained in the image of the map $\Phi$.
Then by Proposition \ref{prp_mua} and the above computation, there exist $a,b,c,d\in\mathbb{Z}_{(5)}$ satisfying the following equations:
\[
\left\{
\begin{array}{ll}
\tfrac{9!}{5!}a+9!b															& =\tfrac{3}{2} \\
\tfrac{13!}{5!5!}a+2\cdot\tfrac{13!}{5!}b+13!c								& =2\\
\tfrac{17!}{5!5!5!}a+3\cdot\tfrac{17!}{5!5!}b+3\cdot\tfrac{17!}{5!}c+17!d	& =2
\end{array}
\right.
\]
But one can find by a slight computation that the denominator of $d$ must be divisible by $125$.
This contradicts the fact that $d\in\mathbb{Z}_{(5)}$.
Thus, the higher Whitehead product $\omega\colon\Sigma^2(A\wedge A\wedge A)\to B\mathrm{G}_2$ is nontrivial.
Therefore $\mathrm{G}_2$ is not a Williams $C_3$-space at $p=5$.
\end{proof}

\bibliographystyle{alpha}
\bibliography{2016commutativity}

\begin{thebibliography}{HKMO18}

\bibitem[Ada69]{MR0251716}
J.~F. Adams.
\newblock Lectures on generalised cohomology.
\newblock In {\em Category {T}heory, {H}omology {T}heory and their
  {A}pplications, {III} ({B}attelle {I}nstitute {C}onference, {S}eattle,
  {W}ash., 1968, {V}ol. {T}hree)}, pages 1--138. Springer, Berlin, 1969.

\bibitem[CS95]{MR1348817}
M.~C. Crabb and W.~A. Sutherland.
\newblock How non-abelian is non-abelian gauge theory?
\newblock {\em Quart. J. Math. Oxford Ser. (2)}, 46(183):279--290, 1995.

\bibitem[Got72]{MR0309111}
Daniel~Henry Gottlieb.
\newblock Applications of bundle map theory.
\newblock {\em Trans. Amer. Math. Soc.}, 171:23--50, 1972.

\bibitem[Hem91]{MR1099785}
Yutaka Hemmi.
\newblock Higher homotopy commutativity of {$H$}-spaces and the mod {$p$} torus
  theorem.
\newblock {\em Pacific J. Math.}, 149(1):95--111, 1991.

\bibitem[HK11]{MR2793107}
Yutaka Hemmi and Yusuke Kawamoto.
\newblock Higher homotopy commutativity and the resultohedra.
\newblock {\em J. Math. Soc. Japan}, 63(2):443--471, 2011.

\bibitem[HKMO18]{MR3775355}
Sho Hasui, Daisuke Kishimoto, Toshiyuki Miyauchi, and Akihiro Ohsita.
\newblock Samelson products in quasi-{$p$}-regular exceptional {L}ie groups.
\newblock {\em Homology Homotopy Appl.}, 20(1):185--208, 2018.

\bibitem[HKO14]{MR3276725}
Sho Hasui, Daisuke Kishimoto, and Akihiro Ohsita.
\newblock Samelson products in {$p$}-regular exceptional {L}ie groups.
\newblock {\em Topology Appl.}, 178:17--29, 2014.

\bibitem[HMR72]{MR0312509}
Peter Hilton, Guido Mislin, and Joseph Roitberg.
\newblock Topological localization and nilpotent groups.
\newblock {\em Bull. Amer. Math. Soc.}, 78:1060--1063, 1972.

\bibitem[IM89]{MR1000378}
Norio Iwase and Mamoru Mimura.
\newblock Higher homotopy associativity.
\newblock In {\em Algebraic topology ({A}rcata, {CA}, 1986)}, volume 1370 of
  {\em Lecture Notes in Math.}, pages 193--220. Springer, Berlin, 1989.

\bibitem[Iwa]{Iwase}
Norio Iwase.
\newblock Associahedra, multiplihedra and units in {$A_\infty$} form.
\newblock {\em arXiv: 1211.5741}.

\bibitem[Iwa98]{MR1642747}
Norio Iwase.
\newblock Ganea's conjecture on {L}usternik-{S}chnirelmann category.
\newblock {\em Bull. London Math. Soc.}, 30(6):623--634, 1998.

\bibitem[KK10]{MR2678992}
Daisuke Kishimoto and Akira Kono.
\newblock Splitting of gauge groups.
\newblock {\em Trans. Amer. Math. Soc.}, 362(12):6715--6731, 2010.

\bibitem[KKT13]{MR3082304}
D.~Kishimoto, A.~Kono, and S.~Theriault.
\newblock Homotopy commutativity in {$p$}-localized gauge groups.
\newblock {\em Proc. Roy. Soc. Edinburgh Sect. A}, 143(4):851--870, 2013.

\bibitem[McG84]{MR745146}
C.~A. McGibbon.
\newblock Homotopy commutativity in localized groups.
\newblock {\em Amer. J. Math.}, 106(3):665--687, 1984.

\bibitem[McG89]{MR999733}
C.~A. McGibbon.
\newblock Higher forms of homotopy commutativity and finite loop spaces.
\newblock {\em Math. Z.}, 201(3):363--374, 1989.

\bibitem[Por65]{MR0174054}
Gerald~J. Porter.
\newblock Higher-order {W}hitehead products.
\newblock {\em Topology}, 3:123--135, 1965.

\bibitem[Sau95]{MR1337215}
Laia Saumell.
\newblock Higher homotopy commutativity in localized groups.
\newblock {\em Math. Z.}, 219(2):203--213, 1995.

\bibitem[Sta63]{MR0158400}
James~Dillon Stasheff.
\newblock Homotopy associativity of {$H$}-spaces. {I}, {II}.
\newblock {\em Trans. Amer. Math. Soc. 108 (1963), 275-292; ibid.},
  108:293--312, 1963.

\bibitem[Sug61]{MR0120645}
Masahiro Sugawara.
\newblock On the homotopy-commutativity of groups and loop spaces.
\newblock {\em Mem. Coll. Sci. Univ. Kyoto Ser. A Math.}, 33:257--269,
  1960/1961.

\bibitem[Tsu16]{MR3491849}
Mitsunobu Tsutaya.
\newblock Mapping spaces from projective spaces.
\newblock {\em Homology Homotopy Appl.}, 18(1):173--203, 2016.

\bibitem[Wat85]{MR807104}
Takashi Watanabe.
\newblock Chern characters on compact {L}ie groups of low rank.
\newblock {\em Osaka J. Math.}, 22(3):463--488, 1985.

\bibitem[Wil69]{MR0240818}
Francis~D. Williams.
\newblock Higher homotopy-commutativity.
\newblock {\em Trans. Amer. Math. Soc.}, 139:191--206, 1969.

\end{thebibliography}
\end{document}